\documentclass[english]{amsart}
\usepackage[T1]{fontenc}
\usepackage[latin9]{inputenc}
\usepackage[active]{srcltx}
\usepackage{babel}
\usepackage{mathtools}
\usepackage{amstext}
\usepackage{amsthm}
\usepackage{amssymb}
\usepackage[unicode=true,
 bookmarks=false,
 breaklinks=false,pdfborder={0 0 1},backref=false,colorlinks=false]
 {hyperref}
\hypersetup{
 colorlinks,linkcolor=blue,anchorcolor=blue,citecolor=blue}
\usepackage{breakurl}

\makeatletter
\numberwithin{equation}{section}
\numberwithin{figure}{section}
\theoremstyle{plain}
\newtheorem{thm}{\protect\theoremname}
  \theoremstyle{plain}
  \newtheorem{lem}[thm]{\protect\lemmaname}
  \theoremstyle{remark}
  \newtheorem{rem}[thm]{\protect\remarkname}

\usepackage{a4wide}
\usepackage{comment}

\usepackage{babel}

  \providecommand{\lemmaname}{Lemma}
  
\providecommand{\theoremname}{Theorem}

\usepackage{babel}
\providecommand{\lemmaname}{Lemma}
  
\providecommand{\theoremname}{Theorem}

\makeatother

  \providecommand{\lemmaname}{Lemma}
  \providecommand{\remarkname}{Remark}
\providecommand{\theoremname}{Theorem}

\begin{document}

\title{singularity of the generator subalgebra in mixed $q$-Gaussian algebras}

\author{Simeng Wang}

\address{Saarland University, Fachbereich Mathematik, Postfach 151150, 66041
Saarbrücken, Germany}

\email{wang@math.uni-sb.de}
\begin{abstract}
We prove that for the mixed $q$-Gaussian algebra $\Gamma_{Q}(H_{\mathbb{R}})$
associated to a real Hilbert space $H_{\mathbb{R}}$ and a real symmetric
matrix $Q=(q_{ij})$ with $\sup|q_{ij}|<1$, the generator subalgebra
is singular.
\end{abstract}

\maketitle
In this note we discuss the generator masas in mixed $q$-Gaussian
algebras. In the early 90's, motivated by mathematical physics and
quantum probability, Bo\.{z}ejko and Speicher introduced the von Neumann
algebra generated by $q$-deformed Gaussian variables \cite{bozejkospeicher91q}.
Since then, this family of von Neumann algebras as well as several
generalizations has attracted a lot of attention. Recently, the generator
subalgebras in these $q$-deformed von Neumann algebras are fruitfully
investigated in \cite{ricard05qfactor,wen17qsingular,skalskiwang2016mixed,bikrammukherjee16qawfactor,bikrammukherjee17qawfactordim3,caspersskalskiwasilewski17qmasaw}.
In this note we will be interested in the case of mixed $q$-Gaussian
algebras introduced in \cite{bozejkospeicher94yangbaxter}, and we
will prove that the associated generator subalgebras are singular.
Our methods are adapted from \cite{skalskiwang2016mixed,wen17qsingular}.

Before the main results let us fix some notation. Let $N\in\mathbb{N}\cup\{\infty\}$,
let $Q=(q_{ij})_{i,j=1}^{N}$ be a symmetric matrix with $q=\max_{i,j}|q_{ij}|<1$,
and let $H_{\mathbb{R}}$ be a real Hilbert space with orthonormal
basis $e_{1},\ldots,e_{N}$. Write $H=H_{\mathbb{R}}+\mathrm{i}H_{\mathbb{R}}$
to be the complexification of $H_{\mathbb{R}}$. Let $\mathcal{F}_{Q}(H)$
be the Fock space associated to the Yang-Baxter operator 
\[
T:H\otimes H\to H\otimes H,\quad e_{i}\otimes e_{j}\mapsto q_{ij}e_{j}\otimes e_{i}
\]
constructed in \cite{bozejkospeicher94yangbaxter}. Let $\Omega$
be the vacuum vector. The left and right creation operators $l_{i}$
are defined by the formulas 
\[
l_{i}\xi=e_{i}\otimes\xi,\quad r_{i}\xi=\xi\otimes e_{i},\quad\xi\in\mathcal{F}_{Q}(H),
\]
and their adjoints $l^{*},r^{*}$ are called the left and right annihilation
operators respectively. We consider the associated mixed $q$-Gaussian
algebra $\Gamma_{Q}(H_{\mathbb{R}})$ generated by the self-adjoint
variables $s_{j}=l_{j}^{*}+l_{j}$. Denote the \emph{Wick product}
map $W:\Gamma_{Q}(H_{\mathbb{R}})\Omega\to\Gamma_{Q}(H_{\mathbb{R}})$
such that $W(\xi)\Omega=\xi$, and the right Wick product map $W_{r}$
on the commutant similarly. Take $\xi_{0}\in H_{\mathbb{R}}$ and
let $M_{\xi_{0}}$ be the von Neumann subalgebra generated by $W(\xi_{0})$
in $\Gamma_{Q}(H_{\mathbb{R}})$. Note that $M_{\xi_{0}}$ is a diffuse
abelian subalgebra. We refer to \cite{bozejkospeicher94yangbaxter,skalskiwang2016mixed}
for any unexplained notation and terminology on the mixed $q$-Gaussian
algebra $\Gamma_{Q}(H_{\mathbb{R}})$. 

Recall that a von Neumann subalgebra $A\subset M$ is called \textit{singular},
if the normalizer $\{u\in\mathcal{U}(M):uAu^{*}=A\}$ is contained
in $A$. For a finite von Neumann algebra $(M,\tau)$, we denote by
$L^{2}(M)$ the completion of $M$ with respect to the norm $\|x\|_{2}^{2}\coloneqq\tau(x^{*}x)$
for any $x\in M$. A subalgebra $A$ is called\textit{ mixing} in
$M$ if for any sequence of unitaries $\{v_{n}\}$ in $A$ which converges
to 0 weakly, we have 
\[
\lim_{n\to\infty}\|\mathbb{E}_{A}(xv_{n}y)-\mathbb{E}_{A}(x)v_{n}\mathbb{E}_{A}(y)\|_{2}=0,\forall x,y\in M,
\]
where $\mathbb{E}_{A}$ stands for the conditional expectation onto
$A$. It is easy to see that for diffuse subalgebras, mixing implies
singularity. We refer to \cite{sinclairsmith08masabook} for more
details on the theory of finite von Neumann algebras. Our results
will be based on the following well-known property (see for example
\cite[Theorem 11.4.1]{sinclairsmith08masabook} and \cite[Proposition 1]{wen17qsingular}.
\begin{lem}
\label{lem:lem}Let $M$ be a finite von Neumann algebra and $A\subset M$
a diffuse subalgebra. Assume that $Y\subset M$ is a subset whose
linear span is $\|\cdot\|_{2}$-dense in $L^{2}(M)$ and $\{v_{n}\}\subset A$
is an orthonormal basis for $L^{2}(A)$. If 
\[
\sum_{n}\|\mathbb{E}_{A}(xv_{n}y)-\mathbb{E}_{A}(x)v_{n}\mathbb{E}_{A}(y)\|_{2}^{2}<\infty,
\]
for all $x,y\in Y$, then $A$ is mixing in $M$. In particular, $A$
is singular in $M$. 
\end{lem}
The following is our main result.
\begin{thm}
Let $v_{n}=\|W(\xi_{0}^{\otimes n})\|_{2}^{-1}W(\xi_{0}^{\otimes n})$,
$n\in\mathbb{N}$. Then for any words $x=W(\xi_{1}\otimes\cdots\otimes\xi_{m})$
and $y=W(\eta_{1}\otimes\cdots\otimes\eta_{k})$ with $\xi_{i},\eta_{j}\in H_{\mathbb{R}}$,
we have
\[
\sum_{n}\|\mathbb{E}_{M_{\xi_{0}}}(xv_{n}y)-\mathbb{E}_{M_{\xi_{0}}}(x)v_{n}\mathbb{E}_{M_{\xi_{0}}}(y)\|_{2}^{2}<\infty.
\]
Consequently, $M_{\xi_{0}}$ is mixing and singular in $\Gamma_{Q}(H_{\mathbb{R}})$.
\end{thm}
It is a standard argument to see from the above theorem that $M_{\xi_{0}}$
is maximal abelian in $\Gamma_{Q}(H_{\mathbb{R}})$ and hence $\Gamma_{Q}(H_{\mathbb{R}})$
is a II$_{1}$ factor if $\dim H_{\mathbb{R}}\geq2$. As a result
we recover the main theorem in \cite{skalskiwang2016mixed}.

Before the proof of the above theorem, let us recall the following
estimate given in \cite[proof of Lemma 1]{skalskiwang2016mixed}.
\begin{lem}
\label{lem:conv general case}Let $(H_{n})_{n\geq1}$ be a sequence
of Hilbert spaces and write $H=\oplus_{n\geq1}H_{n}$. Let $r,s\in\mathbb{N}$
and let $(a_{i})_{1\leq i\leq r}$, $(b_{j})_{1\leq j\leq s}$ be
two families of operators on $H$ which send each $H_{n}$ into $H_{n+1}$
or $H_{n-1}$, such that there exists $0<q<1$ with 
\[
\|[a_{i},b_{j}]|_{H_{n}}\|\leq q^{n},\;\;\;n\in\mathbb{N}.
\]
Assume that $K_{n}\subset H_{n}$ is a finite-dimensional Hilbert
subspace for each $n\geq1$ such that for $K=\oplus_{n}K_{n}$ we
have 
\[
a_{i}(K)\subset K,\quad1\leq i\leq r-1,\quad\text{and }a_{r}|_{K}=0.
\]
Then for any $n\geq1$, there is a constant $C>0$, independent of
$n$, such that 
\[
\|(a_{r}\cdots a_{1}b_{1}\cdots b_{s})|_{K_{n}}\|\leq Cq^{n}.
\]
\end{lem}
\begin{proof}
For each $i$ we may write 
\[
a_{i}b_{1}\cdots b_{s}\xi-b_{1}\cdots b_{s}a_{i}\xi=\sum_{j=1}^{s}b_{1}\cdots b_{j-1}[a_{i},b_{j}]|_{H_{m(j,n)}}b_{j+1}\cdots b_{s}\xi,\quad\xi\in K_{n},
\]
where $m(j,n)$ is an integer greater than $n-s$. Iterating this
formula we obtain 
\begin{align*}
a_{r}\cdots a_{1}b_{1}\cdots b_{s}\xi & =b_{1}\cdots b_{s}a_{r}\cdots a_{1}\xi+\sum_{i=1}^{r}(a_{r}\cdots a_{i}b_{1}\cdots b_{s}a_{i-1}\cdots a_{1}\xi-a_{r}\cdots a_{i+1}b_{1}\cdots b_{s}a_{i}\cdots a_{1}\xi)\\
 & =b_{1}\cdots b_{s}a_{r}\cdots a_{1}\xi+\sum_{i=1}^{r}a_{r}\cdots a_{i+1}\left(\sum_{j=1}^{s}b_{1}\cdots b_{j-1}[a_{i},b_{j}]|_{H_{m'(i,j,n)}}b_{j+1}\cdots b_{s}\right)a_{i-1}\cdots a_{1}\xi,
\end{align*}
where $\xi\in K_{n}$ and for each $i,j,n$ the integer $m'(i,j,n)$
is greater that $n-s-r$. Note that $a_{r}\cdots a_{1}\xi=0$ by the
assumption on $a_{i}$. Since the sum above is independent of $n$,
the lemma is established. 
\end{proof}
Now we may prove our main result.
\begin{proof}[Proof of Theorem 2]
Note that if $x\in M_{\xi_{0}}$ or $y\in M_{\xi_{0}}$, then the
summation is trivially $0$. So without loss of generality we assume
that $\{x\Omega,y\Omega\}\bot\mathcal{F}_{Q}(\mathbb{R}\xi_{0})$.
In this case we have $\mathbb{E}_{M_{\xi_{0}}}(x)=\mathbb{E}_{M_{\xi_{0}}}(y)=0$.
By Lemma \ref{lem:lem}, it is enough to show that $\sum_{n}\|\mathbb{E}_{M_{\xi_{0}}}(xW(\xi_{0}^{\otimes n})y)\Omega\|_{\mathcal{F}_{Q}(H_{\mathbb{R}})}^{2}/\|\xi_{0}^{\otimes n}\|_{\mathcal{F}_{Q}(H_{\mathbb{R}})}^{2}<\infty$.
By the second quantization we know that
\begin{align*}
\mathbb{E}_{M_{\xi_{0}}}(xW(\xi_{0}^{\otimes n})y)\Omega & =P_{\mathcal{F}_{Q}(\mathbb{R}\xi_{0})}(W(\xi_{1}\otimes\cdots\otimes\xi_{m})W(\xi_{0}^{\otimes n})\eta_{1}\otimes\cdots\otimes\eta_{k})\\
 & =P_{\mathcal{F}_{Q}(\mathbb{R}\xi_{0})}(W(\xi_{1}\otimes\cdots\otimes\xi_{m})W_{r}(\eta_{1}\otimes\cdots\otimes\eta_{k})\xi_{0}^{\otimes n}),
\end{align*}
where $P_{\mathcal{F}_{Q}(\mathbb{R}\xi_{0})}$ is the orthogonal
projection from $\mathcal{F}_{Q}(H_{\mathbb{R}})$ to $\mathcal{F}_{Q}(\mathbb{R}\xi_{0})$.
So by the Wick formula in \cite[Theorem 1]{krolak00wickyangbaxter},
it suffices to prove that if 
\[
\zeta_{n}=P_{\mathcal{F}_{Q}(\mathbb{R}\xi_{0})}(l_{i_{1}}\cdots l_{i_{p}}l_{i_{p+1}}^{*}\cdots l_{i_{s}}^{*}r_{j_{1}}\cdots r_{j_{l}}r_{j_{l+1}}^{*}\cdots r_{j_{t}}^{*}\xi_{0}^{\otimes n})
\]
with at least one pair of vectors $\{e_{i_{s'}},e_{j_{t'}}\}\bot\xi_{0}$,
then $\sum_{n\geq0}\|\zeta_{n}\|_{\mathcal{F}_{Q}(H_{\mathbb{R}})}^{2}/\|\xi_{0}^{\otimes n}\|_{\mathcal{F}_{Q}(H_{\mathbb{R}})}^{2}<\infty$.
Take $s''$ to be the largest index in $\{k:e_{i_{k}}\bot\xi_{0}\}$.
Note that we only need to consider the case $s''\geq p+1$ since otherwise
$\zeta_{n}=0$ by orthogonality. Note that it is easy to see that
$\|[l_{i}^{*},r_{j}]|_{H^{\otimes n}}\|\leq q^{n}$ (see e.g. \cite[Lemma 2]{skalskiwang2016mixed}).
Now applying Lemma \ref{lem:conv general case} to the operator $l_{i_{s''}}^{*}\cdots l_{i_{s}}^{*}r_{j_{1}}\cdots r_{j_{l}}r_{j_{l+1}}^{*}\cdots r_{j_{t}}^{*}$,
we see that for all $n$ large enough,
\[
\|\zeta_{n}\|_{\mathcal{F}_{Q}(H_{\mathbb{R}})}\leq Cq^{n}\|\xi_{0}^{\otimes n}\|_{\mathcal{F}_{Q}(H_{\mathbb{R}})},
\]
where $C$ is a constant independent of $n$. Thus $\sum_{n\geq0}\|\zeta_{n}\|_{\mathcal{F}_{Q}(H_{\mathbb{R}})}^{2}/\|\xi_{0}^{\otimes n}\|_{\mathcal{F}_{Q}(H_{\mathbb{R}})}^{2}<\infty$,
as desired. \end{proof}
\begin{rem}
The above arguments can be adapted to the setting of $q$-Araki-Woods
algebras $\Gamma_{q}(H_{\mathbb{R}},U_{t})$ which is recently studied
in \cite{bikrammukherjee17qawfactordim3,bikrammukherjee16qawfactor}.
In particular, this provides a simple proof of the key estimate $\sum_{n\geq0}\|T_{x,y}(H_{n}^{q}(s_{q}(\xi_{0}^{\otimes n}))\Omega)\|_{q}^{2}/\|\xi_{0}^{\otimes n}\|_{q}^{2}<\infty$
in \cite[Lemma 3.1]{bikrammukherjee17qawfactordim3}. According to
the discussion in \cite{bikrammukherjee17qawfactordim3}, the result
is closely related to the factoriality of $\Gamma_{q}(H_{\mathbb{R}},U_{t})$
and yields that $M_{\xi_{0}}$ is a singular masa if $\xi_{0}$ is
invariant under $U_{t}$.
\end{rem}

\subsection*{Acknowledgement}

The author was partially funded by the ERC Advanced Grant on Non-Commutative
Distributions in Free Probability, held by Roland Speicher.

\end{document}